\documentclass[12pt,a4paper]{amsart}
\usepackage{dsfont}
\usepackage{graphicx}
\usepackage{amssymb}
\usepackage[all]{xy}
\usepackage{amsmath}
\usepackage{tikz}
\usepackage[french,english]{babel}
\usepackage{amssymb,amsmath,amsfonts,amsthm,amscd}
\usepackage{indentfirst,graphicx}
\usepackage[font={scriptsize},captionskip=5pt]{subfig}
\usepackage[T1]{fontenc}

\newtheorem{theorem}{Theorem}[section]

\newtheorem{proposition}[theorem]{Proposition}

\newtheorem{definition}[theorem]{Definition}

\newtheorem{remark}[theorem]{Remark}
\newtheorem{example}[theorem]{Example}

 \let\:=\colon  
   \let\?=\overline

\def\and{\hbox{ and }}

\let\:=\colon  
   \let\?=\overline

\let\Sum=\sum \def\sum{\Sum\nolimits}

\def\and{\hbox{ and }}

\def\DONE{*!*}
\def\NextDef #1 {\def\NextOne{#1}
 \ifx\NextOne\DONE\let\next\relax
 \else\expandafter\xdef\csname#1\endcsname{\TheOp}
  \let\next\NextDef
 \fi \next}
\def\TheOp{\mathop{\rm\NextOne}}
 \NextDef 
  Projan Supp Proj Sym Spec Hom cod Ker dist
 *!*

\newcommand{\cH}{{\mathcal H}}
\newcommand{\cI}{{\mathcal I}}

\newcommand{\cO}{{\mathcal O}}

\newcommand{\cR}{{\mathcal R}}

\newcommand{\bC}{{\mathbb{C}}}

\def\TheOp{\hbox{\rm\NextOne}}
\NextDef 
 A ICIS 
 *!*

\begin{document}

\title{THE BI-LIPSCHITZ EQUISINGULARITY OF ESSENTIALLY ISOLATED DETERMINANTAL SINGULARITIES}

\author{Thiago F. da Silva}
\author{Nivaldo G. Grulha Jr.}
\author{Miriam S. Pereira}

\maketitle

\begin{abstract}
{\small The bi-Lipschitz geometry is one of the main subjects in the modern approach of Singularity Theory. However, it rises from works of important mathematicians of the last century, especially Zariski.	In this work we investigate the Bi-Lipschitz equisingularity of families of Essentially Isolated Determinantal Singularities inspired by the approach of Mostowski and Gaffney.}
\end{abstract}

\let\thefootnote\relax\footnote{2010 \textit{Mathematics Subjects Classification} 32S15, 14J17, 32S60

}

\vspace{1cm}

\section*{Introduction}

\vspace{1cm}

The study of bi-Lipschitz equisingularity was started at the end of 1960's with works of Zariski \cite{Za}, Pham \cite{Pham} and Teissier \cite{PT}. At the end of 1980's (see \cite{M1}), Mostowski introduced a new viewpoint for the study of Bi-Lipschitz equisingularity by the existence of Lipschitz stratified vector fields, i.e, every Lipschitz vector field tangent to strata of a stratification can be extended to a Lipschitz vector field defined on the ambient space and tangent to strata. In that work Mostowski has showed that every analytic variety admits a Lipschitz stratification with a such vector field, however, this vector field is not obtained in a canonical way. 

More recently, in \cite{G1}, Gaffney presented conditions ensuring that a family of irreducible curves has a canonical vector field which is Lipschitz, namely

\begin{center}
$\frac{\partial}{\partial y} + \sum\limits_{i}\frac{\partial \tilde{f}_i}{\partial y}\cdot\frac{\partial}{\partial z_i}$
\end{center}

\noindent where $\widetilde{F}:\bC\times \bC \longrightarrow \bC\times \bC^n$ given by $\widetilde{F}(y,t)=(y,\tilde{f}(y,t))$ defines the family of curves. In this case, the main condition is the multiplicity of the pair $e(I_D(\tilde{F}),I_{\Delta})$ to be independent of the parameter $y$, where $I_{\Delta}$ is the ideal defining the diagonal on $\bC\times\bC$ and $I_D(\tilde{F})$ is the ideal generated by the doubles of the components of $\tilde{F}$, defined in section 2, Definition \ref{D2.1}. 

In this work we present conditions which ensure the above vector field is Lipschitz in the context of determinantal varieties. Following the approach of Pereira and Ruas \cite{RP}, for 1-unfoldings   $\tilde{F}:\bC\times\bC^q\longrightarrow \bC\times Hom(\bC^m,\bC^n)$ written as $$\tilde{F}(y,x)=(y,F(x)+y\theta(x))$$

\noindent we show that if $\theta$ is constant then the above vector field is always Lipschitz. If $\theta$ is not constant, we have examples in both cases.

\vspace{1cm}

\section*{Acknowledgements}

The authors are grateful to Terence Gaffney and Maria Aparecida Soares Ruas for the inspiration and support for this work, to David Trotman for his careful reading and valuable comments, to Anne Fr\"uhbis-Kr\"uger, for her comments and suggestions which provided the improvement of this paper, mainly in Theorem 2.7 and by the remark that appears here as Remark 2.8 and to the referee for the excellent suggestions which improved this work.

The first author was supported by Funda\c{c}\~ao de Amparo \`a Pesquisa do Estado de S\~ao Paulo - FAPESP, Brazil, grant 2013/22411-2. The second author was partially supported by Funda\c{c}\~ao de Amparo \`a Pesquisa do Estado de S\~ao Paulo - FAPESP, Brazil, grant 2017/09620-2 and Conselho Nacional de Desenvolvimento Cient\'ifico e Tecnol\'ogico - CNPq, Brazil, grant 303046/2016-3. The third author was supported by Proex ICMC/USP in a visit to S\~ao Carlos, where part of this work was developed.

\vspace{1cm}
 
\section{Determinantal Varieties}

We first recall the definition of determinantal varieties. Let $\Sigma_t\subset Hom(\mathbb{C}^n,\mathbb{C}^{p})$ be the subset
consisting of the maps that have rank less than $t$, with $1\leq t\leq
\min(n,p)$. It is possible to show that $\Sigma_t$ is an irreducible
singular algebraic variety of codimension $(n-t+1)(p-t+1)$ (see
\cite{Bruns}). Moreover the singular set of $\Sigma_t$ is exactly
$\Sigma_{t-1}$. The set $\Sigma_t$ is called a \textit{generic
determinantal variety} of size $(n,p)$ from $t\times t$ minors.

The representation of the variety $\Sigma_t$ as the union $\Sigma_{i}\setminus\Sigma_{i-1}$, $i=1,\,\ldots,t$ is a stratification of $\Sigma_t$, which is locally holomorphically trivial. This is called the \textbf{rank stratification} of $\Sigma_t$.

\begin{definition}
Let $U\subset \mathbb{C}^r$ be an open domain, $F=(m_{ij}(x))$ be a $n\times p$ matrix whose entries are
complex analytic functions on $U$, $0\in U$ and $f$
the function defined by the $t\times t$ minors of $F$. We say that
$X=V(f)$ is a determinantal variety if  it has codimension $(n-t+1)(p-t+1)$.
\end{definition}
 
Currently, determinantal varieties have been an important object of study in Singularity Theory. For example, we can refer to the works of Damon \cite{DP}, Fr\"uhbis-Kr\"uger \cite{FN,FK}, Gaffney \cite{G1,GGR}, Grulha \cite{GGR}, Nu\~no-Ballesteros \cite{ANOT,NOT}, Or\'efice-Okamoto \cite{ANOT,NOT}, Pereira \cite{Miriam,RP}, Pike \cite{DP}, Ruas \cite{GGR,RP}, Tomazella \cite{ANOT,NOT}, Zhang \cite{Zhang} and others.  

In the case where $X$ is a codimension two
determinantal variety, we can use the Hilbert-Burch theorem to
obtain a good description of $X$ and its deformations in terms of
its presentation matrix. In fact, if $X$ is a codimension two Cohen-
Macaulay variety, then $X$ can be defined by the maximal minors of a
$n\times (n+1)$ matrix. Moreover, any perturbation of a  
$n\times (n+1)$ matrix gives rise to a deformation of $X$ and any
deformation of $X$ can be obtained through a perturbation of the
presentation matrix (see \cite{Sch}).
We can use this correspondence to study properties of codimension two  Cohen-Macaulay varieties through their presentation matrix. This is the approach of Fr\"ubis-Kr\"uger, Pereira and Ruas.

In order to introduce the notion of EIDS and relate it with the classical approach of singularity theory, let us recall some concepts in this field.

Let $\cR$ be the group of coordinate changes (on the source) in $(\bC^r,0)$. We denote $GL_i(\cO_r)$ the group of invertible matrices of size $i\times i$ with entries in the local ring $\cO_r$. Consider the group $\cH:=GL_p(\cO_r)\times GL_n(\cO_r)$.

Given two matrices, we are interested in
studying these germs according to the following equivalence relation.

\begin{definition}
Let ${\mathcal{G}}:={\mathcal{R}}\ltimes {\mathcal{H}}$ be the semi-direct product of ${\mathcal{R}}$ and ${\mathcal{H}}$. We say that two germs $F_1,
\,\ F_2\in Mat_{\textit{(n,p)}}({\mathcal{O}}_r)$ are
$\mathcal{G}$-equivalent if there exist $(\phi,R,L)\in
{\mathcal{G}}$ such that $F_1=L^{-1}(\phi^*F_2)R$.
\end{definition}

It is not difficult to see that $\mathcal{G}$ is one of Damon's
geometric subgroups of $\mathcal{K}$ (see \cite{Miriam}), hence as a consequence of Damon's result (\cite{Damon})
we can use the techniques of singularity theory, for instance, those
concerning finite determinacy. The notions of ${\mathcal
{G}}$-equivalence and $ {\mathcal {K}} _ {\Delta} $-equivalence, where $
\Delta $ consists of the subvariety of matrices of rank less than
the maximal rank \cite {Damon}, coincide for finitely determined
germs (see \cite{Bruce}).

The next result is a Geometric Criterion of Finite Determinacy for families of $n\times p$ matrices and was proved in \cite{Miriam}.

\begin{theorem}\label{3.2}\label{detfin}(\textbf{Geometric Criterion of Finite Determinacy})
A representative of a germ $F:\mathbb{C}^r,0\longrightarrow Mat_{(n,p)}(\mathbb{C})$ is
$\mathcal{G}$-finitely determined if and only if $F$ is transverse to
the strata of the rank stratification of $Mat_{(n,p)}(\mathbb{C})$ outside the
origin.
\end{theorem}

It follows that if $F$ is a $n\times p$ matrix with entries in the maximal ideal of ${\mathcal{O}}_r$, defining an isolated singularity, then $F$ is $\mathcal{G}$-finitely determined.
Moreover if $F$ is $\mathcal{G}$-finitely determined, then the germ of $X$ at a singular
point is holomorphic to either the product of $\Sigma_t$ with an affine space or a transverse slice of $\Sigma_t$. This motivates the following definition (\cite{EG}):

\begin{definition} A point $x\in X = F^{−1}(\Sigma_t)$ is called essentially non-singular if, at the point $x$, the map $F$ is transversal to the corresponding stratum of the variety $\Sigma_t$. A germ $(X, 0)\subset (\mathbb{C}^r, 0)$ of a determinantal variety has an essentially isolated singular point at the origin (or is an essentially isolated determinantal singularity: EIDS) if it has
only essentially non-singular points in a punctured neighborhood of the origin in $X$.
\end{definition}

If $X=F^{-1}(\Sigma_t)$ then a
perturbation of $X$ is  obtained by perturbing the entries of $F$. This yields an unfolding of $F$, and if $X$ is an EIDS then happens to also give a deformation of $X$ which is transverse to the strata of $Hom(\mathbb{C}^n,\mathbb{C}^{p})$.

In the particular case where $X=F^{-1}(\Sigma_t)$ is   Cohen-Macaulay of codimension $2$, it is a consequence of the Auslander-Buchsbaum formula and the Hilbert-Burch 
Theorem that any deformation of $X$ can be
given as a perturbation of the presentation matrix (see \cite{FK}, pg 3994). Therefore we can
study these varieties and their deformations using their representation matrices and we can 	express the normal module $N(X)$ in terms of matrices.


%

\section{Determinantal Varieties and Bi-Lipschitz Equisingularity}

Let us recall some definitions and fix some notations. Our main reference for these informations is \cite{G1}.

In this paper we  work with one parameter deformations and unfoldings. The parameter space will be denoted by $Y=\bC\equiv \bC\times 0$.

\begin{definition}\label{D2.1}
Let $h\in\cO_{\bC^N}$. The {\bf double of }$h$ is the element denoted by $h_D\in\cO_{\bC^{2N}}$ defined by the equation $$h_D(z,z'):=h(z)-h(z')$$. 

If $h=(h_1,...,h_r)$ is a map, with $h_i\in\cO_{\bC^N}$, $\forall i$, then we define $I_D(h)$ as the ideal of $\cO_{\bC^{2N}}$ generated by $\{(h_1)_D,...,(h_r)_D\}$.
\end{definition}

\vspace{0,5cm}

Now we get a relation between the integral closure of the double and the property that the canonical vector field induced by a one parameter unfolding be Lipschitz.

Let $\tilde{F}: \bC\times\bC^q\longrightarrow \bC\times\bC^n$ be an analytic map, which is a homeomorphism onto its image, and such that we can write $\tilde{F}(y,x)=(y,\tilde{f}(y,x))$, with $\tilde{f}(y,x)=(\tilde{f}_1(y,x),...,\tilde{f}_n(y,x))$. Let us denote by 

$$\frac{\partial}{\partial y} + \sum\limits_{j=1}^{n}\frac{\partial \tilde{f}_j}{\partial y}\cdot\frac{\partial}{\partial z_j}$$

\noindent the vector field $v: \tilde{F}(\bC\times\bC^q)\longrightarrow \bC\times\bC^n$ given by $$v(y,z)=(1,\frac{\partial \tilde{f}_1}{\partial y}(\tilde{F}^{-1}(y,z)),...,\frac{\partial \tilde{f}_n}{\partial y}(\tilde{F}^{-1}(y,z))).$$

Before stating the result, let us recall the equivalence of the four statements in Theorem 2.1 of \cite{LT}, which goes back to the famous seminars of Teissier and Lejeune given at the \'Ecole Polytechnique in the 1970s on integral dependence in complex analytic geometry.

\begin{theorem}
	Let $(X,\cO_X)$ be a reduced complex analytic space, $x\in X$, and let $\cI$ be a coherent sheaf of $\cO_X-$ideals. Denote $I=\cI_x$ the stalk of $\cI$ on $x$, which is an ideal of $\cO_{X,x}$. Let $h\in\cO_{X,x}$. Suppose that $\cI$ defines a nowhere dense closed subset of $X$. The following are equivalent:
	
	\begin{enumerate}
		\item $h$ is integral over $I$;
		
		\item There exist a neighborhood $U$ of $x$, a positive real number $C$, representatives of the space germ $X$, the function germ $h$, and generators $g_1,...,g_m$ of $I$ on $U$, which we identify with the corresponding germs, so that $$\left\|h(z)\right\|\leq C \max \left\{ \left\|g_1(z)\right\|,...,\left\|g_m(z)\right\|\right\},$$ for all $z\in U$;
		
		\item For all analytic path germs $\phi: (\bC,0)\rightarrow (X,x)$, the pullback $\phi^*(h)$ is contained in the ideal generated by $\phi^*(I)$ in the local ring $\cO_{\bC,0}$;
		
		\item Let $NB$ denote the normalization of the blowup of $X$ by $I$, $\bar{D}$ the pullback of the exceptional divisor of the blowup of $X$ by $I$. Then, for any component $C$ of the underlying set of $\bar{D}$, the order of vanishing of the pullback of $h$ to $NB$ along $C$ is greater than or equal to the order of the divisor $\bar{D}$ along $C$.
	\end{enumerate} 
\end{theorem}

Now we are able to state our result.

\begin{proposition}
The vector field $\frac{\partial}{\partial y} + \sum\limits_{j=1}^{n}\frac{\partial \tilde{f}_j}{\partial y}\cdot\frac{\partial}{\partial z_j}$ is Lipschitz if and only if $$I_D(\frac{\partial \tilde{F}}{\partial y})\subset\overline{I_D(\tilde{F})}.$$
\end{proposition}

\begin{proof}

Since we are working in a finite dimensional $\bC$-vector space then all the norms are equivalent. For simplify the argument, we use the notation $\Vert . \Vert$ for the \textit{maximum norm} on $\bC\times\bC^q$ and $\bC\times\bC^n$, i.e, $\Vert (x_1,...,x_{n+1})\Vert = \max_{i=1}^{n+1}\{\Vert x_i\Vert\}$.

Suppose the canonical vector field is Lipschitz. By hypothesis there exists a constant $c>0$ such that $$\parallel v(y,z)-v(y',z') \parallel \leq c\parallel (y,z)-(y',z') \parallel$$ $\forall (y,z),(y',z')\in U$, where $U$ is an open subset of  $\tilde{F}(\bC\times\bC^q)$.

\vspace{0,3cm}

Thus, given $(y,x),(y',x')\in \tilde{F}^{-1}(U)$, and applying the above inequality on these points, we get $$\parallel (\frac{\partial \tilde{f}_j}{\partial y})_D(y,x,y',x') \parallel \leq c\parallel \tilde{F}(y,x)-\tilde{F}(y',x') \parallel $$ for all $j=1,...n$. By the previous theorem, each generator of $I_D(\frac{\partial \tilde{F}}{\partial y})$ belongs to $\overline{I_D(\tilde{F})}$.

\vspace{0,5cm}

Now suppose that $I_D(\frac{\partial \tilde{F}}{\partial y})\subset\overline{I_D(\tilde{F})}$.  Using the hypothesis and again the Lejeune-Teissier Theorem, for each $j\in\{1,...n\}$ there exists a constant $c_j>0$ and an open subset $U_j\subset\bC\times\bC^q$ such that $$\parallel (\frac{\partial \tilde{f}_j}{\partial y})_D(y,x,y',x') \parallel \leq c_j\parallel \tilde{F}(y,x)-\tilde{F}(y',x') \parallel $$

$\forall (y,x),(y',x')\in U_j$. Take $U:=\bigcap\limits_{j=1}^{n}U_j$, $c:=max\{c_j\}_{j=1}^{n}$ and $V:=\tilde{F}(U)$, which is an open subset of $\tilde{F}(\bC\times\bC^q)$, since $\tilde{F}$ is a homeomorphism onto its image. Hence, $$\parallel v(y,z)-v(y',z') \parallel \leq c\parallel (y,z)-(y',z') \parallel$$ $\forall (y,z),(y',z')\in V$.

Therefore, the vector field $\frac{\partial}{\partial y} + \sum\limits_{j=1}^{n}\frac{\partial \widetilde{f}}{\partial y}\cdot\frac{\partial}{\partial z_j}$ is Lipschitz .

\end{proof}

Now, we have an application to a special case of determinantal varieties.

\begin{proposition}\label{2.2}
Suppose that $\tilde{F}:\bC\times\bC^q\longrightarrow \bC\times Hom(\bC^m,\bC^n)$ is an analytic map and a homeomorphism onto its image, and suppose we can write $$\tilde{F}(y,x)=(y,F(x)+y\theta(x)).$$

\begin{enumerate}
\item [a)] The vector field $\frac{\partial}{\partial y} + \sum\limits_{j=1}^{n}\frac{\partial \tilde{f}_j}{\partial y}\cdot\frac{\partial}{\partial z_j}$ is Lipschitz if, and only if, $$I_D(\theta)\subset\overline{I_D(\tilde{F})}.$$

\item [b)] If $\theta$ is constant then the vector field $\frac{\partial}{\partial y} + \sum\limits_{j=1}^{n}\frac{\partial \tilde{f}_j}{\partial y}\cdot\frac{\partial}{\partial z_j}$ is Lipschitz.
\end{enumerate} 
\end{proposition}

\begin{proof}
(a) It is a straightforward consequence of the last proposition and the equality $\frac{\partial \tilde{f}}{\partial y}=\theta$.

(b) Since $\theta$ is constant then the doubles of the components of $\theta$ are all zero, so $I_D(\theta)$ is the zero ideal, which ensures the inclusion $$I_D(\theta)\subset\overline{I_D(\tilde{F})}.$$
\end{proof}

\begin{remark}
In \cite{ANOT} and \cite{NOT}, the authors consider a one parameter deformation with a constant $\theta$. As showed above, for all these deformations the canonical vector field is Lispchitz.

In Example \ref{E1} we see a case where the deformation does not come from a constant $\theta$, and the canonical vector field remains Lipschitz. In Example \ref{E2} we have another deformation that does not come from a constant $\theta$ where the canonical vector field is not Lipschitz.

As we have seen, the canonical vector field is naturally associated to the 1-unfolding of the variety. However, its behaviour for the Lipschitz equisingularity is not the same. This behaviour depends on the type of the normal form, as we will see later. 

\end{remark}


The first order deformations $T^1_{X}$ can be identified with $\frac{Mat_{\textit{(n,p)}}({\mathcal{O}}_r)}{\mathcal{T}\mathcal{G}F}$, where $\mathcal{T}\mathcal{G}F$ is the extended $\mathcal{G}$-tangent space of the matrix $F$  (Lemma 2.3, \cite{FN}). Hence we can treat the base of the semi-universal deformation using matrix representation and $F$ is $\mathcal{G}$-finitely determined if and only if $T^1_{X}$ is a finite dimensional module. From now on, the element $\theta$ is taken as an element of the space of the first order deformations $T^1_{X}\cong \frac{Mat_{\textit{(n,p)}}({\mathcal{O}}_r)}{\mathcal{T}\mathcal{G}F}$.

\begin{example}\label{E1}
	Consider $$F=\left(
\begin{matrix}
w^l   &  y  &   x   \\
z     &  w  &   y^k  \\
\end{matrix}\right)$$

\noindent with $l,k\geq 2$, which is one of the normal forms obtained in \cite{FN}. Consider the matrix of deformation $$\theta=\left(
\begin{matrix}
\sum\limits_{i=0}^{l-1}w^i   &    0    &     0    \\
0                          &    0    &    \sum\limits_{j=0}^{k-1}y^j
\end{matrix}\right)$$

\noindent and $\tilde{F}(u,x,y,z,w)=(u,F(x,y,z,w)+u\theta(x,y,z,w)$. Notice that $\theta\in\frac{Mat_{(2,3)}(\cO_4)}{T\mathcal{G}F}$. Then, $I_D(\theta)$ is generated by $\{\sum\limits_{i=0}^{l-1}(w^i-w'^i),\sum\limits_{j=0}^{k-1}(y^j-y'^j)\}$. So, the generators are multiples of $(w-w')$ and $(y-y')$, respectively, and these linear differences belong to $I_D(\tilde{F})$. Therefore, $I_D(\theta)\subset I_D(\tilde{F})$. By Proposition \ref{2.2} we conclude that the canonical vector field is Lipschitz.

\end{example}

\begin{example}\label{E2}
	$$F=\left(
	\begin{matrix}
	z   &  y+w^2  &   x^2   \\
	w^k     &  x  &   y  \\
	\end{matrix}\right)$$
	
\noindent with $k\geq 2$, which is one of the normal forms obtained in \cite{FN}. Consider the matrix of deformation	

$$\theta=\left(
\begin{matrix}
0        &    1    &     xw+w    \\
\sum\limits_{i=0}^{k-1}w^i        &  w       &  w  
\end{matrix}\right)$$

\noindent and $\tilde{F}(u,x,y,z,w)=(u,F(x,y,z,w)+u\theta(x,y,z,w)$. Notice that $\theta\in\frac{Mat_{(2,3)}(\cO_4)}{T\mathcal{G}F}$. Then, $I_D(\theta)$ is generated by $\{w-w',w(x+1)-w'(x'+1),\sum\limits_{i=0}^{k-1}(w^i-w'^i)\}$ amd $I_D(\tilde{F})$ is generated by $\{z-z',(y-y')+w^2-w'^2,x^2-x'^2+u(xw+w)-u'(x'w'+w'),w^k-w'^k+u\sum\limits_{i=0}^{k-1}w^i - u'\sum\limits_{i=0}^{k-1}w'^i,x+uw-x'-u'w',y+uw-y'-u'w'\}$.

Consider the curve $\phi:\bC\longrightarrow (\bC\times\bC^4)\times(\bC\times\bC^4)$ given by $\phi(t)=(0,0,0,0,2t,0,0,0,0,t)$.

Then, $I_D(\theta)\circ\phi$ is the ideal of $\cO_1$ generated by $t$, and $I_D(\tilde{F})\circ\phi$ is the ideal generated by $t^2$. Hence, $I_D(\theta)\circ\phi\nsubseteq I_D(\tilde{F})\circ\phi$ and by the curve criterion for the integral closure of ideals we conclude that $I_D(\theta)\nsubseteq\overline{I_D(\tilde{F})}$. Therefore, Proposition \ref{2.2} ensures that the canonical vector field is not Lipschitz.
	
\end{example}

In \cite{FN}, the authors present a classification table for simple isolated Cohen-Macaulay codimension 2 singularities.



\begin{theorem}
	Consider $X$ a variety given by some $F:\bC^q\rightarrow Hom(\bC^3,\bC^2)$ and a semi-universal unfolding $\tilde{F}:\bC\times\bC^{q}\rightarrow \bC\times Hom(\bC^3,\bC^2)$, $q\in\{4,5\}$ as in \ref{2.2}, where $\theta\in\frac{Mat_{(2,3)}(\cO_q)}{T\mathcal{G}F}$.
	Suppose that $X$ is a simple isolated Cohen-Macaulay variety of codimension 2.
	
	If $F$ is of 1-jet-type $J_{q,k}$ from Lemma 3.2 of \cite{FN} then the canonical vector field is Lipschitz, otherwise it is not.
	
\end{theorem}

\begin{proof}
	Suppose that $F$ is of $1$-jet-type from Lemma 3.2 of \cite{FN}. Since $\theta\in\frac{Mat_{(2,3)}(\cO_q)}{T\mathcal{G}F}$ then the $q$ order $1$ entries  of the matrix $F$ stay unperturbed, thus the differences of the monomial generators of the maximal ideal are in $I_D(\tilde{F})$. In particular the ideal $I_\Delta$ from the diagonal satisfies the inclusion $I_\Delta\subseteq I_D(\tilde{F})$. Let $\theta_i$, $i\in\{1,...,6\}$ be the components of $\theta$. Notice that every $(\theta_i)_D$ vanishes on the diagonal $\Delta$ which implies that all the generators of $I_D(\theta)$ belong to $I_\Delta$. Therefore, $I_D(\theta)\subseteq I_\Delta\subseteq I_D(\tilde{F})$ and the Proposition \ref{2.2} ensures that the canonical vector field is Lipschitz. 
	
	Suppose the opposite. In this case, one of the generators  of the maximal ideal is not an entry of the matrix $F$. Without loss of generality, we may assume this is the first coordinate $x$. Since $\tilde{F}$ is a semi-universal unfolding then $x-x'$ certainly appears as a part of a generator set of $I_D(\theta)$. Take the curve $\phi:(\bC,0)\rightarrow (\bC\times \bC^q)\times(\bC\times \bC^q)$ given by $\phi(t)=(0,2t,0,...,0,t,0,...)$. Then $I_D(\tilde{F})\circ\phi$ is generated by $t^m$, for some $m>1$. Since $(x-x')\circ\phi=t$ then $(x-x')\circ\phi\notin I_D(\tilde{F})\circ\phi$, and by the curve criterion we conclude that $x-x'\notin\overline{I_D(\tilde{F})}$. Therefore, $I_D(\theta)\nsubseteq \overline{I_D(\tilde{F})}$ and the Proposition \ref{2.2} ensures that the canonical vector field is not Lipschitz. 
	
\end{proof}

\begin{remark}
We can rephrase the condition on the jet-type by stating:

\begin{enumerate}

\item[a)] {The canonical vector field is Lipschitz if the ideal of 1-minors of the matrix of $X$ defines a reduced point.}
\item[b)] {The canonical vector field is not Lipschitz if the ideal of 1-minors of the matrix of $X$ defines a fat point.}
\end{enumerate}

\end{remark}

\vspace{3cm}

\begin{small}
	
{\sc Thiago F. da Silva
	
	Instituto de Ci\^encias Matem\'aticas e de Computa\'c\~ao - USP \\
	Av. Trabalhador S\~ao Carlense, 400 - Centro, 13566-590 - S\~ao Carlos - SP, Brazil, thiago.filipe@usp.br
	
	Departamento de Matem\'atica, Universidade Federal do Esp\'irito Santo \\
	Av. Fernando Ferrari, 514 - Goiabeiras, 29075-910 - Vit\'oria - ES, Brazil, thiago.silva@ufes.br}
	
\vspace{1cm}	
	
{\sc N. G. Grulha Jr. 
	
	Instituto de Ci\^encias Matem\'aticas e de Computa\'c\~ao - USP\\
	Av. Trabalhador S\~ao Carlense, 400 - Centro, 13566-590 - S\~ao Carlos - SP, Brazil, njunior@icmc.usp.br}

\vspace{1cm}
	
{\sc M. S. Pereira 
	
	Departamento de Matem\'atica, Universidade Federal da Para\'iba, 58.051-900  Jo\~ao Pessoa, Brazil, miriam@mat.ufpb.br}

\end{small}

\end{document}